\documentclass{amsart}
\usepackage{graphicx} 
\usepackage{amssymb}

\usepackage{xcolor}

\usepackage{ulem}

\usepackage{amsthm}
\theoremstyle{plain}
\newtheorem{theorem}{Theorem}[section]

\theoremstyle{definition}
\newtheorem{definition}[theorem]{Definition}
\newtheorem{example}[theorem]{Example}
\newtheorem{conjecture}[theorem]{Conjecture}

\newtheorem{remark}[theorem]{Remark}
\newtheorem{proposition}[theorem]{Proposition}
\newtheorem{lemma}[theorem]{Lemma}
\newtheorem{corollary}[theorem]{Corollary}

\usepackage[
    citestyle=alphabetic,
    style=alphabetic,
    isbn = false,
    url = false
]{biblatex}
\usepackage{hyperref}

\DeclareMathOperator{\Ind}{Ind}
\DeclareMathOperator{\Res}{Res}
\DeclareMathOperator{\Park}{Park}

\newcommand{\C}{\mathbb{C}}
\newcommand{\Z}{\mathbb{Z}}
\newcommand{\trace}{\mathrm{Tr}}
\title{On Extending type $B$ Parking Spaces}
\author{Anthony Adams, Joshua Dorsam, Lily Levitsky, Megan Mann}
\date{\today}

\begin{document}
\begin{abstract}
    Armstrong, Reiner, and Rhoades defined for all Weyl groups $W$ a natural representation of $W$ called the $W$-parking space. The type $B$ parking space is the representation $\mathbb{C}[(\mathbb{Z}/(2n+1)\mathbb{Z})^n]$ of the $n$th signed symmetric group. We consider more general representations of the form $\mathbb{C}[(\mathbb{Z}/m\mathbb{Z})^n]$; we conjecture that this representation extends to the $(n+1)$th signed symmetric group for all $n$ and $m$. We prove this conjecture when $m = 3$ or when $n \leq 2$.
\end{abstract}

\maketitle

\section{Introduction}

In \cite{Haiman}, Haiman defined the parking space $\Park_n$, the representation of the symmetric group $S_n$ spanned by parking functions. It is a rich object of study in algebraic combinatorics with deep connections to Catalan combinatorics. Berget and Rhoades proved that the $S_n-$action on $\Park_n$ in fact extends to an action of $S_{n+1}$ \cite{bergetrhoades14}.
Since then, other extensions of the parking space have been constructed \cite{konvatewari21, KST21}.

Armstrong, Reiner, and Rhoades introduced an analog of the parking space representation for any Coxeter group $W$ \cite{ARR15}.
In this paper, we consider a version of Berget and Rhoades' question for the Coxeter group $B_n,$ the signed symmetric group on $n$ letters. There is an inclusion $B_n\subset B_{n+1}$ given by choosing the elements of $B_{n+1}$ that fix the $(n+1)^{\text{st}} $ letter. We attempt to extend the parking space representation for $B_n$ to $B_{n+1}$ along this inclusion. The type $B$ parking space is the representation $\mathbb{C}[(\Z/(2n+1)\Z)^n]$ spanned by the finite $B_n$-set $(\Z/(2n+1)\Z)^n$.
More generally, we consider the family of $B_n$-representations $\mathbb{C}[(\mathbb{Z}/m\mathbb{Z})^n]$ for all $n,m \geq 1$. Our main conjecture is:
\begin{conjecture}\label{conjecture: main}
    For all $m,n\ge 1,$ there exists a representation $V_{n,m}$ of $B_{n+1}$ such that $$\mathrm{Res}^{B_{n+1}}_{B_n}V_{n,m}\cong \mathbb{C}[(\mathbb{Z}/m\mathbb{Z})^n].$$
\end{conjecture}
In particular, we conjecture that there exists a representation of $B_{n+1}$ whose restriction to $B_n$ is isomorphic to the type $B$ Parking Space. 
\begin{remark}
    It may be evidence towards Conjecture \ref{conjecture: main} that the action of $S_n$ on $(\Z/m\Z)^n$ given by permuting entries is always isomorphic to the restriction of an $S_{n+1}-$set,  for example, the natural action of $S_{n+1}$ on $\{x\in (\Z/m\Z)^{n+1}|\sum x_i =0\}$. 
\end{remark}

\begin{remark}
    It is natural to consider the same generalization for the type $A$ version of the question: does the action of $S_n$ on $\C[ (\Z/m\Z)^{n-1}]$ extend to a representation of $S_{n+1}$ for all $n$? 
    However, this representation does not extend when $m=8$ and $n=11$.
    Note that this does not disprove Conjecture \ref{conjecture: main}.
\end{remark}

While Conjecture \ref{conjecture: main} remains open in full generality, we were able to prove some special cases.
\begin{theorem}\label{theorem: m=3}
    For $m=3$, there exists an extension of $\mathbb{C}[(\mathbb{Z}/m\mathbb{Z})^n]$ to $B_{n+1}$. 
\end{theorem}
\begin{theorem}\label{theorem: n le 2}
    When $n=1$ or $n=2$, there exists an extension of $\mathbb{C}[(\mathbb{Z}/m\mathbb{Z})^n]$ to $B_{n+1}$ for all $m\ge 1.$  
\end{theorem}
Theorems \ref{theorem: m=3} and \ref{theorem: n le 2} are proved in §\ref{section: results}. 
We also used integer linear programming to check Conjecture \ref{conjecture: main} for $n\le 7$ and $ m\le 21$; see §\ref{section: data} for more details.

\subsection*{Acknowledgments}

This paper is based on research conducted at the 2025 MathILy-EST REU. 
The authors thank Joshua Mundinger, the advisor of this project, and Mathematical Staircase, Inc.\ for organizing the program. 
The authors thank Nate Harman for useful comments and conversations.

This material is based upon work supported by the National Science Foundation under Grant Number DMS-2149647. Any opinions, findings, and conclusions or recommendations expressed in this material are those of the authors and do not necessarily reflect the views of the National Science Foundation.

\section{Background}

\subsection{Representations of finite groups}

We must first review some background content that we will reference in our work. 
The following standard facts can be found in sections 1, 2, and 7 of \textit{Linear Representations of Finite Groups} by Serre \cite{Serre77}. 

\begin{definition}
    A (complex) \textit{representation} $\rho$ of a group $G$ on a vector space $V$ over $\C$ is a homomorphism $\rho: G \rightarrow GL(V)$, where $GL(V)$ is the general linear group of $V$, the group of invertible linear transformations on $V$. 
\end{definition}

We can think of representations as a group action by linear transformations on a vector space, namely, for some group $G$ and vector space $V$, $g \cdot v = \rho (g) (v), g \in G, v \in V$. As shorthand, we will refer to a representation by the vector space on which the group acts. Recall that a representation $\rho$ on a vector space $V$ is \textit{irreducible} if there does not exist a nontrivial proper subspace $W \subset V$ which is invariant under the image of $\rho$. The number of complex irreducible representations of a finite group up to isomorphism is equal to the number of conjugacy classes \cite[Theorem 2.7]{Serre77}. 

Every representation has a unique function which associates to each group element the trace of its corresponding linear transformation. 
\begin{definition}
    The \textit{character} $\chi$ of a representation is the function $\chi: G \to \mathbb{C}$ given by $\chi (g) = \trace(\rho(g)) $. 
\end{definition}

Characters are constant over the conjugacy classes of $G$. This relies on the fact that for any linear transformations $A: V \rightarrow V$ and $B: V \rightarrow V$, $\trace(AB)=\trace(BA)$. 

We can then compute the \textit{character table} of $G,$ a matrix with rows indexed by irreducible representations of $G$ and columns indexed by conjugacy classes of $G.$ The entry corresponding to $V$ and the conjugacy class of $g \in G$ is given by $\chi_V(g). $  

Given two representations $V$ and $W$ of $G,$ we define a direct sum representation $V\oplus W$ by letting $G$ act independently on each factor. It turns out that \textit{every } representation of $G$ is the direct sum of irreducible $G-$representations. We can compute their multiplicities by way of the useful identity $\chi_{V\oplus W}= \chi_V+\chi_W. $

We have a number of  ways to relate representations of groups to representations of their subgroups. 

\begin{definition}
    Let $G$ be a group with a subgroup $H \subseteq G$, and let $V$ be a representation of $G$. The \textit{restriction} $\Res^G_H V$ from $G$ to $H$ is the representation of $H$ on $V$ defined only on elements $h \in H$.
\end{definition}

\begin{definition}
    Let $G$ be a group with a subgroup $H \subseteq G$, and let $V$ be a representation of $H$. Let $g_1, ..., g_k$ be coset representatives for $G/H$. The \textit{induced representation} from $H$ to $G$ is the vector space $\Ind_H^G V =  \bigoplus_{i=1}^{k} g_iV$, the direct sum of copies of V labeled by $g_i$. To define the action of $G$ on this space, we declare that if $g g_i = g_j h$ for some $h \in H$, then $g (g_iv) = g_j(hv) \in g_j V$, the copy of V labeled by $g_j$. 
\end{definition}

In this paper, we are specifically concerned with a type of representation called the \textit{linearization} of a $G-$set. 

\begin{definition}\label{linearization}
    For any group $G$ acting on a set $X,$ we define $\C[X]$  to be a vector space with basis elements labeled by elements of $X$. Equivalently, $\C[X]$ is the complex vector space containing formal complex linear combinations of elements in $X$. The action of $G$ on $X$ defines a $G-$representation on $\C[X]$.
\end{definition}

The specific sets we are concerned with will be related to the action of the symmetric group on tensor powers:

\begin{definition}
    Suppose $V$ is a complex vector space. The $n^{\text{th}}$ \textit{tensor power} of $V$ has underlying vector space $V^{\otimes n}$ with an action of the $n$th symmetric group $S_n$ given by permuting the tensor factors. Explicitly, $\sigma\in S_n$ acts on pure tensors by 
    $$\sigma\cdot (v_1\otimes\dots\otimes v_n)= v_{\sigma^{-1}(1)}\otimes \cdots v_{\sigma^{-1}(n)}.$$
 \end{definition}

\subsection{Representation Theory of the Symmetric Group }

The irreducible representations of the symmetric group have a particularly useful classification, which we will use to find decompositions of $B_n$-representations.
Refer to Chapter 1 of \textit{Representation Theory of the Symmetric Group} \cite{jameskerber84} for further background on the Symmetric Group.

\begin{definition}
A \textit{partition} $\lambda$ of a positive integer $n$ is a sequence of weakly decreasing nonnegative 
integers $(\lambda_1 \geq ... \geq  \lambda_k)$ that sum to $n$.
\end{definition}
We will write $\lambda \vdash n$ to indicate that $\lambda$ is a partition of $n$ and define $|\lambda| = \sum_i \lambda_i$. We will also adopt the convention that $\emptyset \vdash 0$ is the only partition of 0.

We can use partitions to classify irreducible representations of the symmetric group. Recall that every element of the symmetric group can be decomposed into a unique product of disjoint cycles. 

\begin{definition}
    Suppose $\sigma \in S_n$ decomposes into the product of disjoint cycles $c_1, c_2,...,c_k$ with orders $a_1,\dots, a_k.$ The \textit{cycle type} of $\sigma$ is the partition $\lambda \vdash n$ equal to the weakly decreasing reordering $\lambda = (a_1' \geq a_2' \dots \geq a_k')$. 
\end{definition}

There exists a bijection between irreducible representations of $S_n$ and partitions $\lambda\vdash n.$ We will write the irreducible representation of $S_n$ corresponding to the partition $\lambda$ as $L_{\lambda}$, as in \cite[35-36]{jameskerber84}.

Remarkably, the following theorem characterizes all tensor spaces of the symmetric group in terms of these partitions.

\begin{theorem}[Schur-Weyl Duality, \cite{jameskerber84}, 4.3.2] Let $V$ be a $\mathbb{C}$-vector space with $d := \dim V < \infty$. Then there is an isomorphism of $S_n$-representations  $$V^{\otimes n} \cong \bigoplus_{\lambda \vdash n:\, \ell(\lambda) \leq d} \mathbb{S}^\lambda V \boxtimes L_\lambda
  $$
where $\mathbb{S}^{\lambda}V$ is the subspace of $V^{\otimes n}$ obtained by applying the Young symmetrizer corresponding to $\lambda$.
Thus, as an $S_n$-module,
\[
V^{\otimes n}\ \cong\ \bigoplus_{\substack{\lambda\vdash n\\ \ell(\lambda)\le d}}
(\dim \mathbb{S}^\lambda V)\,L_\lambda,
\]
 
\end{theorem}

The details of the construction of $\mathbb{S}^{\lambda}V$ are outside of this paper, but it turns out that we can explicitly calculate the coefficients $\dim(\mathbb{S}^{\lambda}V)$ of $L_{\lambda}$ with the following result.  

\begin{theorem}[The Weyl Dimension Formula, \cite{FultonHarris}, Theorem 6.3]\label{weyldimformula}

Let $k=\dim V.$ Then \( \mathbb{S}^{\lambda}V \) is zero if \( \lambda_{k+1} \neq 0 \). Otherwise, we have the formula
\[
\dim \mathbb{S}^{\lambda}V = \prod_{1 \leq i < j \leq k} \left(\frac{\lambda_i - \lambda_j + j - i}{j - i}\right).
\]

\end{theorem}

These results give an explicit expression for the multiplicities of irreducible representations of $S_n$ in tensor space.

\begin{corollary}\label{finalschurweyl}
Let $k=\dim V$. Then $V^{\otimes n}$ has the following decomposition into irreducible representations of $S_n$:

\[	V^{\otimes n} \cong \bigoplus_{\lambda\vdash n} d(\lambda,k)L_\lambda,\]
where 
$$d(\lambda,a)=\begin{cases} \frac{\prod_{i < j \leq a} (\lambda_i - \lambda_j+j-i)}{\prod_{i<j \leq a}{(j-i)}} & \text{rows}(\lambda)\le a \\
    0 & \text{rows}(\lambda)>a.
    \end{cases}$$

\end{corollary}

\subsection{The Signed Symmetric Group}

\begin{definition}
     The \textit{signed symmetric group} on $n$ letters, notated $B_n$, is the group of all $n \times n $ matrices with exactly one nonzero entry in each row and column that is either 1 or -1. 
\end{definition}

Equivalently, $B_n$ can be thought of as the wreath product of $S_2$ (negation) and $S_n$ (permutation). $B_n$ is also known as the hyperoctahedral group, as it is the symmetry group of a hypercube in $n$ dimensions. 

There exists a homomorphism $\pi : B_n \rightarrow S_n$ which sends every nonzero entry of a matrix in $B_n$ to 1 and considers the resulting permutation matrix in $S_n$. We define $\beta \in B_n$ to be a cycle in $B_n$ if and only if $\pi(\beta)$ is a cycle in $S_n$. Similar to the symmetric group, every element of the signed symmetric group can be decomposed uniquely into a product of disjoint cycles.

\begin{definition}
    A cycle $\beta \in B_n$ is \textit{even} if the number of -1's in its corresponding matrix is even. Otherwise, the cycle is \textit{odd}.
\end{definition}

Separating elements of $B_n$ into products of odd and even cycles leads to a clean characterization of the conjugacy classes of $B_n$. 

\begin{definition}
A \textit{bipartition} of a positive integer $n$ is an ordered pair of partitions $\lambda, \mu$ such that $|\lambda| + |\mu| = n$. 

\end{definition}
We will write $(\lambda, \mu) \vDash n$ to indicate that $(\lambda, \mu)$ is a bipartition of $n$.

\begin{definition}

    Consider $\beta =  \beta_+ \beta_- \in B_n$ where $\beta_+$ decomposes into the disjoint product of even cycles and $\beta_-$ decomposes into the disjoint product of odd cycles. We define the \textit{signed cycle type} of $\beta$ to be the bipartition $(\lambda, \mu) \vDash n$, where $\lambda$ is the cycle type of $\pi(\beta_+)$ and $\mu$ is the cycle type of $\pi(\beta_-)$.

\end{definition}

\begin{example}
    We can correspond $$\begin{pmatrix}
    0&1&0&0&0\\
    1&0&0&0&0\\
    0&0&0&-1&0\\
    0&0&1&0&0\\
    0&0&0&0&1\\

    \end{pmatrix}\in B_5$$ to the  permutation $(1,2),(3,4,-3,-4)\in B_5$ with signed cycle type $((1),(1)).$
\end{example} 

\begin{proposition}
 [\cite{jameskerber84}, Theorem 4.2.8] Suppose $\beta_1,\beta_2\in B_n.$ Then $\beta_1$ and $\beta_2$ are conjugate if and only if they have the same signed cycle type. 
\end{proposition}

Parallel to the symmetric group case, there is an explicit bijection between irreducible representations of $B_n$ and bipartitions $(\lambda,\mu)\vDash n.$ Additionally, these irreducible representations of $B_n$ can be built from the irreducible representations of $S_n. $

Let $\epsilon$ be the \textit{signed character} of $B_n$, the function that assigns to a signed permutation matrix $\beta \in B_n$ the product of its nonzero entries. Now, let $(\lambda, \mu) \vDash n$, with $|\lambda| = l$ and $|\mu| = n-l$, and let $L_{\lambda}, L_{\mu}$ be the irreducible representations of $S_l$ and $S_{n-l}$ corresponding to the partitions $\lambda, \mu$, respectively. We define 
\begin{definition}\label{V_{lambda,mu}}
$$V_{\lambda, \mu} = \Ind_{B_{l} \times B_{n-l}}^{B_n} L_{\lambda} \boxtimes (L_{\mu} \otimes \epsilon). $$
\end{definition}

It turns out that each such $V_{\lambda, \mu}$ is an irreducible representation of $B_n$, and we can build all irreducible representations of $B_n$ this way. 

\begin{theorem}[\cite{jameskerber84}, 4.3.34] \label{specificSchurWeyl}
 The map $(\lambda,\mu)\to V_{\lambda,\mu}$ is a bijection between bipartitions of $n$ and irreducible representations of $B_n$ up to isomorphism. 
\end{theorem}

The final step in our characterization of $B_n$-representations is to consider the behavior of our irreducible representations under restriction.

\begin{theorem}[The Branching Rule for type $B$, \cite{Zelevinsky1981}, 7.6] 

Let $(\lambda,\mu)\vDash n$ and let $R_{\lambda, \mu}$ be the set of bipartitions obtained by removing one box from the Young diagram of $(\lambda, \mu)$. Then we have 
$$\mathrm{Res}_{B_{n-1}}^{B_n}  V_{\lambda, \mu}  = 
\bigoplus_{(\lambda', \mu')  \ \in \ R_{\lambda, \mu}} V_{\lambda', \mu'} $$ 
\end{theorem}

\section{Results}\label{section: results}
\subsection{Character formulas and irreducible decompositions}
We now have the tools to compute the irreducible decompositions of $\C[(\Z/m\Z)^n].$ It turns out that the decompositions differ depending on the parity of $m$ because $\Z/m\Z$ has only one element which is fixed under negation when $m$ is odd, but two such elements when $m$ is even.

\begin{definition}
    Consider the action of $B_n$ on the set $(\mathbb{Z}/(2k+1)\mathbb{Z})^n$ given by permuting and negating entries. We define the \textit{odd generalized parking space} to be the linearization of this action as in Definition \ref{linearization}:  $$OP_{n,k}=  \mathbb{C}[(\mathbb{Z}/(2k + 1)\mathbb{Z})^n]$$
    
\end{definition}
\begin{remark}
    \cite{ARR15} associated to each Weyl Group $W$ a natural representation called its $W$-parking space. For  $W=B_n,$ the type $B$ parking space is $OP_{n, n} = \mathbb{C}[(\Z/(2n+1)\Z)^n].$ 
\end{remark}

We can define an analogous representation for even $m.$

\begin{definition}
For any natural number $\ell$, consider the group action of $B_n$ on $(\mathbb{Z}/(2\ell)\mathbb{Z})^n$ given by permuting and negating entries. The \textit{even generalized parking space} is the linearization of this action: 
   $$ EP_{n,\ell} = \mathbb{C}[(\mathbb{Z}/2\ell\mathbb{Z})^n]$$

\end{definition}
For $\sigma\in B_n,$ let $\ell_{odd}(\sigma)$ be the number of odd cycles in $\sigma$, and let $\ell_{even}(\sigma)$ be the number of even cycles in $\sigma$. We can find character formulas for the representations defined above in terms of these values. 

\begin{lemma}\label{lemma: odd character formula}
   Let $\chi_{n,k}$ be the character of $OP_{n,k}.$ Then for $\sigma\in B_n,$ 
   $$\chi_{n,k}(\sigma)= (2k+1)^{\ell_{even}(\sigma)}.$$
    
\end{lemma}
\begin{proof}
  By definition, $OP_{n,k} = \mathbb{C}[(\mathbb Z/(2k+1)\mathbb Z)^n]$ as a $B_n-$representation. Since $\sigma\in B_n$ acts by permuting basis elements of this linearization, its matrix representation is a permutation matrix. Hence, the trace is simply the number of ones on the diagonal, or equivalently, the number of basis vectors which $\sigma $ fixes.   
  
  Suppose $(a_1,\dots, a_n)\in OP_{n,k}$ is a fixed point of $\sigma.$ Then if $\sigma(i)=j,$ we have $a_i=a_j,$ and if $\sigma(i)=-j,$ then $a_i=-a_j.$ Also, an odd number of sign flips occurs in an odd cycle, and hence $a_i=-a_i$ for each $i$ in an odd cycle. Hence, we must have $a_i=0$ for each $i$ in an odd cycle. However,  any even cycle is conjugate to a cycle which does not negate any elements, so we simply have $a_i=a_j$ for all $i$ and $j$ in $\sigma.$ Therefore, we obtain all fixed points by assigning an element of $\Z/(2k+1)\Z $ to each even cycle in $\sigma$, so there are $(2k+1)^{\ell_{even}(\sigma)}$ fixed points, giving the desired formula. 
\end{proof}
We can use a similar argument for the even generalized parking space. 
\begin{lemma}\label{lemma: even character formula}
  Let $\psi_{n,\ell}(\sigma)$ be the character of $EP_{n,\ell}$. Then for $\sigma\in B_n,$ 
  $$\psi_{n,\ell}(\sigma)= 2^{\ell_{odd}(\sigma)}(2\ell)^{\ell_{even}(\sigma)}.$$
\end{lemma}
\begin{proof}
    Again, recall that $EP_{n,\ell} = \mathbb{C}[(\mathbb Z/(2\ell)\mathbb Z)^n]$ as a $B_n-$representation. Let $(a_1\dots, a_n)\in EP_{n,\ell}$ be a fixed point of $\sigma\in B_n$. As in the odd parking space, all indices in an even cycle must be equal of any value. However, both $0$ and $\ell$ are fixed under negation, so the indices in an odd cycle can either be all $0$ or all $\ell.$ This gives $2^{\ell_{odd}(\sigma)}(2\ell)^{\ell_{even}(\sigma)}$ fixed points.
\end{proof}

Recall by Theorem \ref{specificSchurWeyl} that irreducible representations  of $B_n$ can be indexed by bipartitions of $n$: 
$$V_{\lambda, \mu} = Ind_{B_{l} \times B_{m}}^{B_n} L_{\lambda} \boxtimes (L_{\mu} \otimes \epsilon)$$
where $L_{\lambda}$ is the irreducible representation of $S_n$ associated with $\lambda$. We use Schur-Weyl duality to define a decomposition of $OP_{n,k}$ in terms of this labeling of irreducible representations. 
\begin{proposition}\label{OPdecomp}
   $$
        OP_{n, k} \cong \bigoplus_{(\lambda,\mu) \vDash n} d(\lambda, k+1) d(\mu, k) V_{\lambda,\mu}
   $$
    where the sum is over bipartitions $(\lambda,\mu)$ of $n$ where $\lambda$ has at most $k+1$ rows and $\mu$ has at most $k$ rows, and the Weyl dimension formula gives
    $$d(\lambda,a)=\begin{cases} \frac{\prod_{i < j \leq a} (\lambda_i - \lambda_j+j-i)}{\prod_{i<j \leq a}{(j-i)}}. & \text{rows}(\lambda)\le a \\
    0 & \text{rows}(\lambda)>a\end{cases}$$
\end{proposition}
\begin{proof}

For $n=1$, we have $\chi_{1,k}(1)=2k+1, \chi_{1,k}(-1)=1.$ Since the only two irreducible representations of $B_1\cong S_2$ are the trivial and signed representations, $OP_{1,k}\cong \mathbb{C}[\mathbb{Z}/(2k+1)\mathbb{Z}]$ decomposes into $k+1$ copies of the trivial representation and $k$ copies of the signed representation.

Let $W_+$ be a $k+1$-dimensional vector space and $W_-$ be a $k$-dimensional vector space, so that $\mathbb{C}[\mathbb{Z}/(2k+1)\mathbb{Z}] \cong W_+ \oplus (W_- \otimes \epsilon)$ as $B_1\cong S_2$ representations, where $\epsilon$ is the signed character. Then we have an isomorphism of $B_n-$representations:
\begin{align*}
	OP_{n,k} &\cong  (W_+ \oplus (W_- \otimes \epsilon))^{\otimes n} \\
    	&\cong  \bigoplus_{i=0}^n \mathrm{Ind}_{B_i \times B_{n-i}}^{B_n} (W_+)^{\otimes i} \otimes (W_- \otimes \epsilon)^{\otimes n-i}. 
\end{align*}

We can apply Corollary \ref{finalschurweyl} to decompose $(W_+^{\otimes i})$ as an $S_i-$representation and $W_-^{\otimes n-i}$ as an $S_{n-i}-$representation. 
This gives the decomposition 
$$ OP_{n,k}\cong \bigoplus_{i=0}^{n}\bigoplus_{\lambda\vdash i,\mu\vdash n-i} d(\lambda,k+1)d(\mu,k) \mathrm{Ind}_{B_{i}\times B_{n-i}}^{B_n}L_{\lambda} \otimes (L_{\mu}\otimes \epsilon^{\otimes n-i}). $$

 Also, $\epsilon^{\otimes n-i}$ is the signed representation for $B_{n-i}$, so by definition of $V_{\lambda,\mu}$ (Definition \ref{V_{lambda,mu}}),
 
 $$  OP_{n, k} \cong \bigoplus_{(\lambda,\mu) \vDash n} d(\lambda, k+1) d(\mu, k) V_{\lambda,\mu}.\qedhere$$

\end{proof}
\begin{example} \label{ex:WeylOP}
    When $k=1$, we consider only bipartitions $(\lambda,\mu)$ such that $\lambda$ has at most two rows and $\mu$ has at most one row. We will denote bipartitions of this form as $((\lambda_1, \lambda_2),a)$ for some nonnegative integer $a$. Here, the Weyl dimension formula simplifies to $\lambda_1+1-\lambda_2$, so we have 
    \begin{equation}
    OP_{n,1}\cong \bigoplus_{i=0}^n \bigoplus_{\lambda \vdash i} (\lambda_1 - \lambda_2 + 1)V_{((\lambda_1, \lambda_2),n-i)}.
    \end{equation}
\end{example}
We can find a similar decomposition for the generalized even parking space by following the same procedure. 
\begin{proposition}\label{epdecomp }
   $EP_{n,\ell }=\bigoplus_{(\lambda,\mu)\vDash n} d(\lambda, \ell+1)d(\mu,\ell -1)V_{\lambda,\mu}$
     where the sum is over bipartitions $(\lambda,\mu)$ of $n$ where $\lambda$ has at most $\ell +1$ rows and $\mu$ has at most $\ell-1$ rows, and the Weyl dimension formula gives
    $$d(\lambda,a)=\begin{cases} \frac{\prod_{i < j \leq a} (\lambda_i - \lambda_j+j-i)}{\prod_{i<j \leq a}{(j-i)}}. & \text{rows}(\lambda)\le a \\
    0 & \text{rows}(\lambda)>a\end{cases}$$
\end{proposition}
\begin{proof}
Using our character formula, we see that $\psi_{1,\ell }(1)=(2L)^12^0=2L$ and $\psi_{1,\ell }(-1)= (2\ell )^02^1=2.$ We still have that the trivial and signed characters are all the irreducible characters of $B_1,$ so we must have $\ell +1$ copies of the trivial representation and $\ell -1$ copies of the signed representation. Hence, if $W_+$ is an $\ell +1$ dimensional vector space and $W_-$ is an $\ell -1 $ dimensional vector space we have 
$$ EP_{1,\ell }\cong  W_+\oplus (W_-\otimes \epsilon).$$
Then by the same argument as in Proposition \ref{OPdecomp},
\begin{align*}
    EP_{n,\ell }&\cong  \left( W_+\oplus W_-\otimes\epsilon \right)^{\otimes n} \\
    &\cong \bigoplus_{i=0}^n \mathrm{Ind}_{B_i\times B_{n-i}}^{B_n} W_+^{\otimes i} \otimes (W_-\otimes \epsilon)^{\otimes n}\\
    &\cong \bigoplus_{(\lambda,\mu)\vDash n}d(\lambda,\ell +1)d(\mu,\ell-1) L_{\lambda}\otimes(L_{\mu}\otimes \epsilon^{n-i})\\
    &\cong \bigoplus_{(\lambda,\mu)\vDash n}d(\lambda,L+1)d(\mu,L-1)V_{\lambda,\mu}.\qedhere
\end{align*}
\end{proof}
\subsection{Construction of extensions when $m$ is fixed}

We were able to find extensions from $B_n$ to $B_{n+1}$ of $OP_{n,1}= \C[(\Z/3\Z)^n]$ for all $n.$ We find it useful to restrict attention to a subset of irreducible  $B_{n+1}-$representations. 

\begin{definition}
Define $X_{n,m}$ to be the set of bipartitions $(\lambda,\mu)\vDash n$ such that the coefficient of $V_{\lambda,\mu}$ is nonzero in the decomposition of $\C[(\Z/m\Z)^n]$ into irreducible representations of $B_n.$ 

If $\lambda= (\lambda_1,\dots ,\lambda_m)$, set $\lambda' = (\lambda_1+1 ,\lambda_2,\dots, \lambda_m).$ Then define
$$\tilde{X}_{n,m}= \{ (\lambda',\mu)| (\lambda,\mu)\in X_{n,m}\}.$$

\end{definition}
If $m = 2k + 1$ is odd, then recall from Proposition \ref{OPdecomp} that $X_{n,m}$ contains exactly the  bipartitions of $n$ such that $\mathrm{rows}(\lambda)\le k+1$ and $\mathrm{rows}(\mu)\le k$. Similarly, if $m = 2\ell$ is even then recall from Proposition \ref{epdecomp } that $X_{n,m}$ contains exactly the bipartitions of $n$ such that $\mathrm{rows}(\lambda)\le \ell +1$ and $\mathrm{rows}(\mu)\le \ell -1$.

In the case $m=3, k=1$ and $\tilde{X}_{n,3}$ determines a natural choice of irreducible $B_{n+1}-$representations to construct an extension: for each irreducible representation $V_{\lambda,\mu}$, we have that $V_{\lambda,\mu}$ appears in the restriction of $V_{\lambda',\mu}$ by the branching rule. We show that some linear combination of irreducible representations corresponding to $\tilde{X}_{n,3} $ suffices to extend $OP_{n, 1}.$ 
\begin{theorem} \label{thm:OPn1Thm}
For all $n,$ the $B_{n+1}-$representation
$$\bigoplus_{(\lambda,a)\in X_{n,m}}  (\lceil \frac{\lambda_1 + 1 - \lambda_2}{3} \rceil ) V_{(\lambda',a)}$$
restricts to $OP_{n, 1}.$ In fact, it is the only $B_{n+1}-$extension of $OP_{n,1}$ which decomposes into irreducibles indexed by $\tilde{X}_{n,3}.$ 
\end{theorem}

\begin{proof}
Recall by proposition \ref{OPdecomp} that 
$$  
         OP_{n,1}\cong \bigoplus_{i=0}^n \bigoplus_{\lambda \vdash i} (\lambda_1 - \lambda_2 + 1)V_{((\lambda_1, \lambda_2),n-i)}, 
   $$

so for  $k=1 $ ($m=3$), we have that 
$$X_{n,3}=\{ ((\lambda_1,\lambda_2),a)\vDash n \}$$ 
and 
$$\tilde X_{n,3}= \{(\lambda',a)| (\lambda,a)\in X_{n,3}\}.$$

 Thus, finding an extension to $B_{n+1}$ comprised of irreducible representations corresponding to $\tilde X_{n,3}$ is equivalent to finding coefficients  $c_{\lambda' }$ satisfying the following equation:
\begin{equation}\label{eq: extension problem for OPn1}
    \bigoplus_{i=0}^n \bigoplus_{\lambda \vdash i} (\lambda_1 - \lambda_2 + 1)V_{((\lambda_1, \lambda_2),n-i)} \cong \bigoplus_{(\lambda, a) \in X_n} c_{\lambda'}\mathrm{Res}_{B_n}^{B_{n+1}} V_{\lambda',a}.
\end{equation}   
We will find a recursive formula for $c_{\lambda'}$ in terms of $\lambda_1-\lambda_2,$ as only $\lambda_1-\lambda_2$ determines the coefficient of $V_{\lambda,a}$ in the Weyl dimension formula.

First suppose $\lambda_1=\lambda_2.$ Then the only irreducible representation corresponding to  $\tilde{X}_{n,3}$ whose restriction contains $V_{\lambda,a}$ is $V_{\lambda',a},$ since partitions  $(\mu,b)\in \tilde X_{n,1}$ have $(\mu_1>\mu_2)$ by construction. Hence, $$c_{\lambda'}= \lambda_1-\lambda_2+1=1.$$

Suppose $\lambda_1-\lambda_2=1.$ There are exactly two bipartitions in $\tilde X_{n,3}$ such that removing one box yields $(\lambda,a): (\lambda',a)$ and $(\lambda,a+1).$ Hence, 
\begin{align*}
    2 &= \lambda_1-\lambda_2+1\\
    &= c_{\lambda}+c_{\lambda'}.
\end{align*}
But recall that $\lambda_1-1=\lambda_2,$ so by the previous case,
$$ c_{\lambda}= c_{(\lambda_1-1,\lambda_2)'}=1.$$ Therefore, 
$$c_{\lambda'}=1.$$ 

Suppose $\lambda_1-\lambda_2\ge 2.$ Then a box can be added to either row of $\lambda$ to obtain a partition in $\tilde X_{n,3}.$ Hence there are three bipartitions in $\tilde X_{n,3}$ such that removing one box yields $(\lambda,a):$ $(\lambda',a), ((\lambda_1,\lambda_2+1),a),$ and $(\lambda,a+1).$ This gives the recursive formula  

\begin{equation} 
\lambda_1-\lambda_2 +1 = c_{\lambda'}+ c_{(\lambda_1-1,\lambda_2)'} + c_{(\lambda_1-1,\lambda_2+1)'} .
\label{eq:recursion}
\end{equation}
 It follows by induction that if $\lambda_1-\lambda_2=\sigma_1-\sigma_2$, then $c_{\lambda'}=c_{\sigma'}.$ Therefore, for any nonnegative integer $a$ we can define $c_a=c_{\lambda'}$ for any partition $\lambda$ with $\lambda_1-\lambda_2+1=a. $ For $\lambda$ with $\lambda_1-\lambda_2\ge 2$ transforms the recursion above into the form 

 $$c_a= \begin{cases} 1 & a\le 2\\
    a- c_{a-1}-c_{a-2} &\text{rows}(\lambda)>2
\end{cases}$$
This recurrence relation has a unique solution $c_a=\lceil \frac{a}{3}\rceil.$  Therefore, we indeed have the desired decomposition

\begin{equation*} OP_{n,1}= \bigoplus_{(\lambda,a)\in X_{n,m}}  (\lceil \frac{\lambda_1 + 1 - \lambda_2}{3} \rceil ) V_{(\lambda',a)}.
\qedhere
\end{equation*}
\end{proof}

\subsection{Construction of extensions when $n$ is fixed}

We found extensions from $B_n$ to $B_{n+1}$ of $\C[(\Z/m\Z)^n]$ when $n\le 2$ for all $m.$ We fix $n$ instead of $m$ and create a restriction matrix $R_{n, m}$ that encodes the data of the restriction of irreducible representations labeled by elements of $\tilde{X}_{n,m}$. To do so, we will index the rows and columns of $R_{n, m}$ with elements of $X_{n,m}$ and $\tilde{X}_{n,m}$ in an order that ensures invertibility.

\begin{definition}
    Let $\lambda= (\lambda_1,\dots,\lambda_k)$ and $\mu= (\mu_1,\dots, \mu_k) $ be partitions. Then, we say $\lambda> \mu$ in \textit{lexicographical order} iff $\lambda_i>\mu_i$ when $i$ is the least index such that $\lambda_i \neq \mu_i.$ 

    Let $(\lambda^{(1)},\mu^{(1)})$ and $(\lambda^{(2)},\mu^{(2)})$ be bipartitions of $n.$ Then we say  $(\lambda^{(1)},\mu^{(1)})> (\lambda^{(2)},\mu^{(2)})$ if $\lambda^{(1)}>\lambda^{(2)}$ in lexicographical order, or $\lambda^{(1)}=\lambda^{(2)}$ and $\mu^{(1)}>\mu^{(2)}.$
   
\end{definition}
\begin{definition}
   Let us enumerate $X_{n,m}$ and $\tilde{X}_{n,m}$ in reverse lexicographical order, from greatest to smallest. Let $R_{n,m}$ be the square matrix with rows indexed by $X_{n,m}$ and columns indexed by $\tilde{X}_{n,m}$. We define the entries of $R_{n,m}$ as follows: $a_{(\lambda,\mu),(\sigma,\tau)}=1$ if $V_{\lambda,\mu}\subseteq \mathrm{Res}^{B_{n+1}}_{B_n}V_{\sigma,\tau}$ and $a_{(\lambda,\mu),(\sigma,\tau)}=0$ otherwise. 

\end{definition}

\begin{example}
    Let $n=2,m=3$. In reverse lexicographical order, we have 
    $$ X_{2,3}= \{((2,0),\emptyset), ((1,1),\emptyset),((1,0),1),(\emptyset, 2)\}$$
    and 
    $$\tilde{X}_{2,3}= \{((3,0),\emptyset), ((2,1),\emptyset),((2,0),1),((1,0), 2)\}.$$ This gives 
    $$ R_{2,3}= \begin{pmatrix}
        1 & 1& 1& 0\\
        0&1&0& 0 \\
        0&0&1&1\\
        0&0&0&1
    \end{pmatrix}.
    $$
\end{example}

\begin{proposition}
    $R_{n,m}$ is an upper triangular matrix with ones on the diagonal for all positive integers $n$ and $m.$
\end{proposition}
\begin{proof}
We aim to show that $R_{n,m}$ is upper triangular with all ones on the diagonal. 
First, observe that the bijection $f: X_n\to \tilde{X}_n$ given by $\lambda \mapsto \lambda'$  preserves lexicographical order. Then since $V_{\lambda,\mu}\subset \mathrm{Res}^{B_{n+1}}_{B_n}V_{\lambda',\mu}$ for all $(\lambda,\mu)\vDash n,$ we have $a_{(\lambda,\mu),(\lambda',\mu)}=1$ for all $(\lambda,\mu)\vDash n.$

Now suppose that $V_{\lambda,\mu}\subset \mathrm{Res}^{B_{n+1}}_{B_n}V_{\sigma',\tau}$ for some bipartitions $(\lambda,\mu)$ and $(\sigma,\tau)$ of $n.$ Then either $\lambda=\sigma'$ and $\mu$ is $\tau$ with a box removed, or $\mu=\tau$ and $\lambda $ is $\sigma'$ with a box removed. In the first case, $\lambda_1$ has more boxes than $\sigma_1$, so $\lambda>\sigma$ in lexicographical order. Therefore, the corresponding one is above the diagonal. 

In the second case, observe that removing a box from the first row is necessarily less in lexicographical order than any other way of removing a box from $\sigma'.$ Hence, $\sigma \leq \lambda,$ so the corresponding entry is on or above the diagonal.  

Therefore, all ones are on or above the diagonal.
\end{proof}
It follows from linear algebra that the matrix $R_{n,m}$ is always invertible. 
We can use the inverse-restriction matrix to find extensions of $\C[(\Z/m\Z)^n] $ consisting of irreducible representations corresponding to $\tilde X_{n,m}.$ suppose $|X_{n,m}|= a.$ Label the standard basis vectors $e_1,\dots, e_a$ of $\Z^a$ by $\tilde X_{n,m}$ in reverse lexographical order. Call this span $\tilde M. $ Similarly, we can label the standard basis vectors $e_1,\dots, e_a$ of $\Z^a$ by $X_{n,m},$ calling this space $M.$ Then $R_{n,m}$ defines an invertible linear transformation $\tilde M\to M$ where we have 
$$ R_{n,m}\cdot e_{(\lambda,\mu)}= \sum_{V_{\sigma,\tau}\subset \mathrm{Res}^{B_{n+1}}_{B_n}V_{\lambda,\mu }}e_{\sigma,\tau}\in M.$$
Extending by linearity, any vector $v \in \tilde M$ with nonnegative entries corresponds to some $B_{n+1}$-representation $V, $ and  $R_{n,m}\cdot v$ corresponds to $\mathrm{Res}^{B_{n+1}}_{B_n} V$ by linearity. Hence, for $v\in M $ corresponding to a $B_n$-representation $V,$ $R_{n,m}^{-1}\cdot v$ is a vector of nonnegative integers if and only if $V$ extends to a $B_{n+1}$-representation with irreducible representations indexed by $\tilde{X}_{n,m}$. If such an extension exists, $R_{n,m}^{-1}(v)$ gives its irreducible decomposition.  We can use this method to classify the existence or non-existence of such extensions in general.  
 
\begin{theorem}
$OP_{1,k} $ extends to a representation of $B_2$ for all $k \geq 1.$  This is the unique extension such that each irreducible corresponds to $\tilde{X}_{1,m}$ where $m=2k+1.$
\end{theorem}

\begin{proof}
For $n=1,$ we have that $X_{1,m} = \{ ((1), \emptyset) , (\emptyset, (1)) \}$ and $\tilde{X}_{1,m} = \{ ((2), \emptyset), (1, 1) \}$. We can remove a box from $((2), \emptyset )$ and obtain $((1), \emptyset ) \in X_1$, and we can remove a box from $((1), (1))$ and obtain either element of $X_{1,m}$. Thus
$$
R_{1,m}=\begin{pmatrix}
1 & 1 \\
0 & 1
\end{pmatrix}, \qquad 
R_{1,m}^{-1}= \begin{pmatrix}
1 & -1 \\
0 & 1
\end{pmatrix}.
$$
To identify the resulting vector $v$, we use the Weyl dimension formula (Theorem \ref{weyldimformula}) to compute  

$$d(\emptyset,k)=1, \; d((1),k)=k.$$

Therefore, Proposition \ref{OPdecomp} gives the coefficients $$a_{((1),\emptyset)}= d((1),k+1)d(\emptyset,k)= k+1$$ and $$a_{(\emptyset,(1,0))}= d(\emptyset,k+1)d((1),k)=k,$$

so the vector corresponding to $OP_{1,k}$ is $$v=(k+1,k).$$ We can now compute 
$$ R_{1,m}^{-1}\cdot v= (1,k).$$ Both entries are positive, so $V_{(2),\emptyset}\oplus kV_{(1),(1)}$ is the unique extension of $OP_{1,k}$ to $B_2$ with only irreducibles corresponding to $\tilde X_{1,m}$.
\end{proof} 
\begin{theorem}
     $EP_{1,\ell}$ extends to a representation of $B_2$  for  all $\ell.$ This is the unique extension that decomposes into irreducible representations corresponding to $\tilde{X}_{1, m}$ where $m=2\ell .$
\end{theorem}
   
\begin{proof}
Just as in the odd case, we have $X_{1,m} = \{ ((1), \emptyset) , (\emptyset, (1)) \}$ and $\tilde{X}_{1,m} = \{ ((2), \emptyset), (1, 1) \}$. This yields the same restriction matrix 
$$
    R_{1,m}= \begin{pmatrix}1&1\\0&1\end{pmatrix},
    \qquad R_{1,m}^{-1}= \begin{pmatrix}1&-1\\0&1\end{pmatrix}.
$$
It also remains true that  $d(1,k)=k, d(\emptyset,k)= 1$ Then $EP_{1,\ell}$ corresponds to the vector $v= (\ell+1,\ell-1),$ so 
$R_{1,m}^{-1}\cdot v= (2,\ell-1).$ This vector has nonnegative entries, so it corresponds to the valid extension 
$$2V_{((2),\emptyset)}\oplus (\ell-1)V_{(1),(1)}.
$$
\end{proof}

\begin{theorem}
$OP_{2,k}$ extends to a representation of $B_3$ for all $k \geq 1$. This is the unique extension such that each irreducible corresponds to $\tilde{X}_{n,m}$ where $m=2k+1.$
\end{theorem}
\begin{proof}
For $k=1$, we already know that such an extension with $\tilde{X}_{2,3}$ exists. For each $k\ge 2$, we have the same irreducibles, since two boxes cannot fill more than two rows. In reverse lexicographical order,  we have 
$$X_{n,m}= \{ ((2,0),\emptyset),((1,1),\emptyset),((1,0),1),(\emptyset, 2),(\emptyset,(1,1))\},$$
$$\tilde X_{n,m}= \{((3,0),\emptyset),((2,1),\emptyset),((2,0),(1,0)),((1,0),(2,0)),((1,0),(1,1))\}.$$

and hence the same restriction matrix: 
$$ R_{2,m}= \begin{pmatrix}
1 & 1 & 1 & 0 & 0 \\
0 & 1 & 0 & 0 & 0 \\
0 & 0 & 1 & 1 & 1 \\
0 & 0 & 0 & 1 & 0 \\
0 & 0 & 0 & 0 & 1 \\
\end{pmatrix}, \quad R_{2,m}^{-1}= \begin{pmatrix}
1 & -1 & -1 & 1 & 1 \\
0 & 1 & 0 & 0 & 0 \\
0 & 0 & 1 & -1 & -1 \\
0 & 0 & 0 & 1 & 0 \\
0 & 0 & 0 & 0 & 1 \\
\end{pmatrix}.$$ 
Now we can compute the coefficients of each irreducible given by the Weyl dimension formula.  We already have that $d(\emptyset,k)=1$ and $d((1,0),k)=k.$ Also, 
$$d((2,0),k)= \frac{k(k+1)}{2}$$
and 
$$d((1,1),k)= \frac{k(k-1)}{2}.$$

Hence, $OP_{2,k}$ corresponds to the vector
$$v= \left( \frac{(k+1)(k+2)}{2},\frac{k(k+1)}{2}, k(k+1), \frac{k(k+1)}{2}, \frac{k(k-1)}{2}\right)$$
The second, fourth, and fifth entries of $R_{2,m}^{-1} \cdot v$ are single entries of $v$ which are positive as $k \geq 1$.
The first entry of $R_{2,m}^{-1}\cdot v$ is 
$$\frac{(k+1)(k+2)}{2} - k(k+1) + \frac{k(k-1)}{2} = 1 > 0$$ 
and the third entry is 
$$ k(k+1)- \frac{k(k+1)}{2}- \frac{k(k-1)}{2} = k^2 >0.$$

All entries in $v$ are positive, so this method constructs a valid extension as claimed.
\end{proof}

The strategy of only considering irreducible representations of $B_{n+1}$ corresponding to bipartitions in $\tilde X_{n,m}$ does not generalize beyond $n=1$ for $m$ even and $n=2$ for $m$ odd, as the candidate extension vector given by $R_{n,m}^{-1}\cdot v$ does not always have nonnegative coefficients. For example, $EP_{2,4}$ does not admit such an extension.

For $EP_{2,\ell}$, we were able to find an extension by hand by solving the system of linear equations given by the branching rule. 

\begin{proposition}
    $EP_{2,\ell }$ extends to a representation of $B_3$ for all $k.$
\end{proposition}
\begin{proof}

Consider the $B_3$-representation
\begin{align*}
     V_{\ell}&= \frac{(\ell-1)(\ell-2)}{2} V_{1,(1,1)}\oplus \frac{\ell(\ell-1)}{2}V_{1,2}\oplus 2(\ell-1)V_{(1,1),1}\oplus 2(\ell-1)V_{(2),(1)}\\ &\quad \oplus\left(\frac{\ell(\ell+1)}{2}-2(\ell-1)\right)V_{(2,1),\emptyset}\oplus (\ell+1)V_{3,\emptyset}
\end{align*}

Note that $\frac{\ell(\ell+1)}{2}-2(\ell-1)$ is nonnegative for positive integers $\ell, $ so $V_{\ell}$ is a valid representation. We claim $V_{\ell}$ is an extension of $EP_{2,{\ell}}$ for all positive integers $\ell.$ 

Let $W_m=\mathrm{Res}^{B_3}_{B_2} V_{\ell}=\bigoplus a_{(\lambda,\mu)}V_{(\lambda,\mu)}.$ Then \begin{itemize}
\item $a_{\emptyset,(1,1)}= \frac{(\ell-1)(\ell-2)}{2}= d(\emptyset, \ell+1)d((1,1),\ell-1)$
\item $a_{\emptyset, (2)}=\frac{\ell(\ell-1)}{2}= d(\emptyset,\ell+1)d((2),\ell-1).$ 
\item $a_{(1),(1)}= \frac{(\ell-1)(\ell-2)}{2}+\frac{\ell(\ell-1)}{2}+2(\ell-1)= (\ell-1)(\ell+1)=d((1),\ell+1)d((1),\ell-1)$
\item $a_{(1,1),\emptyset}= 2(\ell-1)+\frac{\ell(\ell+1)}{2}-2(\ell-1)= \frac{\ell(\ell+1)}{2}= d((1,1),\ell+1)d(\emptyset,\ell-1)$
\item$  a_{(2),\emptyset}= 2(\ell-1)+\frac{\ell(\ell+1)}{2}-2(\ell-1)+ \ell+1= \frac{(\ell+1)(\ell+2)}{2}= d((2),\ell+1)d(\emptyset,\ell-1)$
\end{itemize}
Therefore, we have that $W_m= EP_{2,\ell}$ and we are done. 
\end{proof}

\section{Data}\label{section: data}

For particular $m$ and $n$ it is possible to use Integer Linear Programs (ILP) to determine whether or not $\C[(\Z/m\Z)^n]$ extends to a $B_{n+1}$-representation. 
\subsection{Implementing ILP}
Suppose $\chi_1,\dots, \chi_k$ are the irreducible characters of $B_{n+1}$ and $\psi$ is the character of $\C[(\Z/m\Z)^n].$ Then the existence of an extension is equivalent to the existence of a set of nonnegative integers $c_1,\dots, c_k$ such that 
$$\psi = \bigoplus_{1\le i\le k} \mathrm{Res}^{B_{n+1}}_{B_n} c_i \chi_i.$$

 Solving for these $c_i$ under the inequalities $c_i\ge 0$ is an instance of integer linear programming. To implement ILP, we must convert this problem into a system of linear equations with imposed inequalities. A conjugacy class of $B_{n+1}$ restricts to a conjugacy class of $B_n$ if it fixes at least one value of $\{1,\dots, n\}$. 
 
 The first step is to generate the restricted character table of $B_{n+1}$, which is just the columns of the $B_{n+1}$ character table represented by a conjugacy class that restricts to $B_n$. 
 The second step is to use the character formulas given by Lemma \ref{lemma: odd character formula} ($m$ odd) and Lemma \ref{lemma: even character formula} ($m$ even) to calculate the character of the $B_n$ action on $\C[(\Z/m\Z)^n].$

 The last step is to use ILP to check if there is a linear combination of all rows of the restricted character table that adds to the character of $\C[(\Z/m\Z)^n]$. If such a linear combination exists, then so does an extension.

The main obstacle with ILP is that the runtime becomes prohibitive rather quickly, making it difficult to check larger examples.
\subsection{Data Collected}
This method was used to confirm the following extensions:
\begin{itemize}
    \item The existence of an extension of $OP_{n,k}$ was confirmed for all pairs $(n,k)$ such that $1 \leq n \leq 7$ and $1 \leq k \leq 10$. Note that this includes an extension of $BP_n$ for $1 \leq n \leq 7$.
    \item The existence of an extension of $EP_{n,\ell}$ was confirmed for all pairs $(n, \ell)$ such that $1 \leq n \leq 7$ and $1 \leq \ell \leq 10$.
\end{itemize}

\printbibliography

\end{document}